\documentclass[12pt]{amsart}
 
\usepackage{amsmath, amsfonts, amssymb, amsthm, cite, amscd, verbatim, comment,fullpage}

\usepackage{tikz}
\usetikzlibrary{matrix,arrows,decorations.pathmorphing}

\numberwithin{equation}{section}

\makeatletter
\newtheorem*{rep@theorem}{\rep@title}
\newcommand{\newreptheorem}[2]{%
\newenvironment{rep#1}[1]{%
 \def\rep@title{#2 \ref{##1}}%
 \begin{rep@theorem}}%
 {\end{rep@theorem}}}
 \makeatother

\theoremstyle{plain}
\newtheorem{theorem}{Theorem}[section]
\newtheorem*{theoremun}{Theorem}
\newreptheorem{theorem}{Theorem}
\newtheorem{lemma}[theorem]{Lemma}
\newtheorem{corollary}[theorem]{Corollary}

\newtheorem{proposition}[theorem]{Proposition}
\theoremstyle{definition}
\newtheorem{definition}[theorem]{Definition}

\theoremstyle{remark}
\newtheorem*{remark}{Remark}

\newcommand{\Gl}{\operatorname{GL}}
\newcommand{\Sl}{\operatorname{SL}}
\newcommand{\defi}{\stackrel{\text{\tiny def}}{=}}
\newcommand{\degree}{\operatorname{deg}}
\newcommand{\Tr}{\operatorname{Tr}}
\newcommand{\N}{\operatorname{N}}

\newcommand{\End}{\operatorname{End} }

\begin{document}

\title{On the trace and norm maps from $\Gamma_0(\mathfrak{p})$ to $\Gl_2(A)$}
\author{Christelle Vincent}
\address{Department of Mathematics, Stanford University, California 94305}
\email{cvincent@stanford.edu}
\keywords{Drinfeld modular forms, Drinfeld modules, congruences}

\maketitle

\begin{abstract}
Let $f$ be a  Drinfeld modular form for $\Gamma_0(\mathfrak{p})$. From such a form, one can obtain two forms for the full modular group $\Gl_2(A)$: by taking the trace or the norm from $\Gamma_0(\mathfrak{p})$ to $\Gl_2(A)$. In this paper we show some connections between the arithmetic modulo $\mathfrak{p}$ of the coefficients of the $u$-series expansion of $f$ and those of a form closely related to its trace, and of the coefficients of $f$ and those of its norm.
\end{abstract}

\section{Introduction and statement of results}

For a prime $\ell \in \mathbb{Z}$, reduction modulo $\ell$ connects modular forms for the congruence subgroup $\Gamma_0(\ell)$ to modular forms for the full modular group $\Sl_2(\mathbb{Z})$. More precisely, we have the following two theorems:

In \cite{serre}, Serre shows the following:

\begin{theoremun}
There is a one-to-one correspondence between forms of weight $2$ for $\Gamma_0(\ell)$ with rational $\ell$-integral $q$-series coefficients and forms of weight $\ell+1$ for $\Sl_2(\mathbb{Z})$ with rational $\ell$-integral $q$-series coefficients. Furthermore, define
\begin{equation*}
\Tr(f) = \sum_{\gamma \in \Gamma_0(\ell)\backslash \Sl_2(\mathbb{Z})} f | [\gamma],
\end{equation*}
then the correspondence is given by the map
\begin{equation*}
f \equiv \Tr(fg_{(0)}) \pmod{\ell}
\end{equation*}
where $g_{(0)}$ is an auxiliary form such that $g_{(0)} \equiv 1 \pmod{\ell}$. 
\end{theoremun}

In \cite[Proposition 5.2]{ahlgrenpapa}, Ahlgren and Papanikolas generalize Lemma 3 from \cite{ahlgrenono} and show the following:

 \begin{theoremun}
 Suppose that $f  \in S_k(\Gamma_0(\ell))$ has rational, $\ell$-integral $q$-series coefficients, with leading coefficient 1, and that $f | [W_{\ell}] = \pm f$, where $[W_{\ell}]$ is the Fricke involution. Let $\widetilde{\N(f)}$ be the multiple of
 \begin{equation*}
 \N(f) = \prod_{\gamma \in \Gamma_0(\ell)\backslash \Sl_2(\mathbb{Z})} f | [\gamma]
 \end{equation*}
 that has leading coefficient $1$. Then $\widetilde{\N(f)} \in S_{k(\ell+1)}(\Sl_2(\mathbb{Z}))$, $\widetilde{\N(f)}$ has rational, $\ell$-integral $q$-series coefficients, and
 \begin{equation*}
\widetilde{\N(f)} \equiv f^2 \pmod{\ell}.
\end{equation*}
\end{theoremun}

The Drinfeld setting offers for function fields constructions analogous to elliptic curves, modular forms, and modular curves in the classical setting. It is therefore natural to wonder if analogous theorems hold in the Drinfeld setting, and in this paper we show that they do.

For the remainder of this work, we now let $q$ be a power of a prime $p$, rather than the classical $e^{2\pi i z}$ above. For $\mathbb{F}_q$ the finite field of order $q$, we set $A=\mathbb{F}_q[T]$ and $K=\mathbb{F}_q(T)$. We will fix a monic prime polynomial $\pi(T) \in A$, of degree $d$, and denote by $\mathfrak{p}$ the ideal generated by this polynomial. Then we have:

\begin{theorem}\label{tracetheorem}
Let $q\geq 3$. There is a one-to-one correspondence between forms of weight $2$ and type $1$ for $\Gamma_0(\mathfrak{p})$ with rational $\pi$-integral $u$-series coefficients and forms of weight $q^d+1$ and type $1$ for $\Gl_2(A)$ with rational $\pi$-integral $u$-series coefficients. Furthermore, define
\begin{equation*}
\Tr(f) = \sum_{\gamma \in \Gamma_0(\mathfrak{p})\backslash \Gl_2(A)} f | [\gamma],
\end{equation*}
then the correspondence is given by the map
\begin{equation*}
f  \equiv \Tr(fg_{(0)}) \pmod{\mathfrak{p}}
\end{equation*}
where $g_{(0)}$ is an auxiliary form such that $g_{(0)} \equiv 1 \pmod{\mathfrak{p}}$. 
\end{theorem}

And also

\begin{theorem}\label{normtheorem}
Let $f$ be a Drinfeld modular form for $\Gamma_0(\mathfrak{p})$ with integral $u$-series coefficients, and such that the leading coefficient of its $u$-series expansion is 1. Suppose further that $f$ is an eigenform of the Fricke involution. Let $\widetilde{\N(f)}(z)$ be the multiple of
 \begin{equation*}
 \N(f) =  \prod_{\gamma \in \Gamma_0(\mathfrak{p})\backslash \Gl_2(A)} f | [\gamma]
 \end{equation*}
 that has leading coefficient $1$. Then $\widetilde{\N(f)}$ has integral $u$-series coefficients and
\begin{equation*}
\widetilde{\N(f)} \equiv f^2 \pmod{\mathfrak{p}}.
\end{equation*}
\end{theorem}

As in Serre's classical work on this subject, Theorem \ref{tracetheorem} is a corollary of a more general result, which we give here since it is of independent interest.

\begin{reptheorem}{tracemap}
Let $f$ be a Drinfeld modular form of weight $k$ and type $l$ for $\Gamma_0(\mathfrak{p})$, with rational $u$-series coefficients. Then $f$ is a $\mathfrak{p}$-adic Drinfeld modular form for $\Gl_2(A)$.
\end{reptheorem}

This paper is organized as follows: In Section 2 we present basic facts on Drinfeld modular forms to establish our notation. In Section 3, we perform some computation to get formulae and integrality results for some operators we will need. In Section 4 we prove Theorem \ref{tracemap}, from which we will obtain Theorem \ref{tracetheorem}. Finally, in Section 5 we show Theorem \ref{normtheorem}.

\section*{Acknowledgements}

Some of this work is part of the author's PhD thesis, and the author is grateful to her adviser Ken Ono, who suggested the problem. The author also thanks Brian Conrad and Samit Dasgupta for helpful conversations. Finally, the author thanks the referees for their careful reading of the paper and thoughtful suggestions for improvement.

\section{Preliminaries and notation}

Recall from the Introduction that we set $A = \mathbb{F}_q[T]$, where $q$ is a power of a prime $p$, and $K = \mathbb{F}_q(T)$. We complete $K$ at the infinite place $v_{\infty}(x)=-\operatorname{deg}(x)$, and write $K_{\infty}=\mathbb{F}_q(\!(1/T)\!)$ for the completion of $K$ at this place. We will also write
\begin{equation*}
C=\hat{\bar{K}}_{\infty}
\end{equation*}
for the completed algebraic closure of $K_{\infty}$, and $\Omega=\mathbb{P}^1(C)-\mathbb{P}^1(K_{\infty})=C - K_{\infty}$. $\Omega$ has a rigid analytic structure described in \cite{gekelerjacobian}, and we call it the Drinfeld upper half-plane. The group $\Gl_2(A)$ acts on $\Omega$ by fractional linear transformations.

\subsection{The Carlitz module}

Let $\End_{C,\mathbb{F}_q}(\mathbb{G}_a)$ be the ring of $\mathbb{F}_q$-linear morphisms of the additive group scheme $\mathbb{G}_a$ that are defined over the field $C$. Then $\End_{C,\mathbb{F}_q}(\mathbb{G}_a)$ is the ring of polynomials of the form
\begin{equation*}
\sum a_i x^{q^i}, \qquad a_i \in C,
\end{equation*}
where the ring multiplication is composition.

We will need Carlitz's module $\rho$ of rank 1, first studied by Carlitz in~\cite{carlitz4}. It is the ring homomorphism
\begin{align*}
\rho \colon A &\to \End_{C,\mathbb{F}_q}(\mathbb{G}_a) \\
a &\mapsto \rho_a
\end{align*}
which sends $T$ to the polynomial
\begin{equation}\label{carlitz}
\rho_T(x) = T x+ x^q.
\end{equation}

One can show that there is a unique rank $1$ $A$-lattice $L$, and a rigid analytic, entire, surjective, $\mathbb{F}_q$-linear, $L$-periodic function $e_L \colon C \to C$ given by the product
\begin{equation*}
e_{L}(z) = z\prod_{\substack{\lambda \in L \\ \lambda \neq 0}}\left(1-\frac{z}{\lambda}\right)
\end{equation*}
such that for each $a \in A$, the following diagram commutes:
\begin{equation*}
\begin{CD}
0 @>>> L @>>> C @>e_{L}>> C  @>>> 0 \\
 @.   @VV \cdot a  V    @V V  \cdot_a  V       @V V \rho_a V\\
0 @>>> L  @>>> C @>e_{L}>> C @>>> 0, 
\end{CD}     
\end{equation*}
where $\cdot a$ denotes the usual multiplication by $a$ on $C$.

We write $L=\tilde{\pi}A$, where the \emph{Carlitz period} $\tilde{\pi}\in K_{\infty}(\sqrt[q-1]{-T})$ is defined up to a $(q-1)$th root of unity. We choose one such $\tilde{\pi}$ and fix it for the remainder of this work.

For $\pi \in A$ a monic prime polynomial of degree $d$ as in the Introduction, we have
\begin{equation}\label{rho}
\rho_{\pi}(x) =\pi x +\sum_{1\leq i \leq d-1}\alpha_i x^{q^i}+x^{q^d}
\end{equation}
with each $\alpha_i$ in $A$. 

We have:

\begin{lemma}[Hayes, \cite{hayes}, Proposition 2.4]
For any positive integer $n$ and prime $\pi \in A$, the polynomial
\begin{equation*}
\frac{\rho_{\pi^n}(x)}{\rho_{\pi^{n-1}}(x)} \in A[x]
\end{equation*} 
is Eisenstein.
\end{lemma}

\begin{corollary}\label{congprop}
If $\pi \in A$ is prime, $\frac{\rho_{\pi}(x)}{x}$ is irreducible, with coefficients in $A$. Furthermore,
\begin{equation*}
\rho_{\pi}(x) \equiv x^{q^d} \pmod{\mathfrak{p}},
\end{equation*}
where $\mathfrak{p}$ is the ideal generated by $\pi$.
\end{corollary}

\begin{proof}
That $\rho_{\pi}(x)$ has integral coefficients is clear from its definition, from which it follows that $\frac{\rho_{\pi}(x)}{x}$ also does. Since $\rho_{1}(x)=x$, the other assertions follow from Hayes's Lemma above.
\end{proof}

\subsection{Drinfeld modular forms}

In this section, when we refer to rigid analytic objects we will mean rigid analytic in the sense of \cite{fresnel}. For a more complete reference on Drinfeld modules and Drinfeld modular forms, we refer the reader to Gekeler's excellent \emph{Inventiones} paper \cite{gekeler}, or to the author's PhD thesis \cite{vincentthesis}. 

Recall that we use $\mathfrak{p}$ to denote a prime ideal of $A$. For any such ideal, we define
\begin{equation*}
\Gamma_0(\mathfrak{p})\defi \left\{
\begin{pmatrix}
a & b \\
c & d
\end{pmatrix}
\in \Gl_2(A) \mid c \equiv 0 \pmod{\mathfrak{p}} \right\},
\end{equation*}
a subgroup of $\Gl_2(A)$.

Let $\Gamma$ be $\Gl_2(A)$ or $\Gamma_0(\mathfrak{p})$. We may take the quotient $\Gamma \backslash \Omega$, and this space has a rigid analytic structure inherited from that of $\Omega$. Furthermore, there is a smooth affine algebraic curve $M_{\Gamma}$ defined over $C$ such that $\Gamma \backslash \Omega$ is canonically isomorphic to the rigid analytic space associated to $M_{\Gamma}$. The curve $M_{\Gamma}$ can be compactified; in the case where $\Gamma=\Gl_2(A)$ this is done by adding one cusp denoted $\infty$, and in the case where $\Gamma=\Gamma_0(\mathfrak{p})$ this is done by adding two cusps denoted by $0$ and $\infty$. We denote the compactified curve by $X_0(\mathfrak{p})$. An analytic parameter at $\infty$ is given by the analytic function
\begin{equation*}
u(z) = \frac{1}{e_L(\tilde{\pi}z)},
\end{equation*}
where $L$ is the lattice associated to the Carlitz module above, and $\tilde{\pi}$ is the Carlitz period. (This is not completely true: an analytic parameter is given by $u(z)^{q-1}$, but since we will deal with Drinfeld modular forms with type they will have expansions in $u$.)

\begin{definition}\label{modform}
Let $\Gamma$ be $\Gl_2(A)$ or $\Gamma_0(\mathfrak{p})$. A function $f \colon \Omega \rightarrow C$ is called a \emph{Drinfeld modular form of weight $k$ and type $l$ for $\Gamma$}, where $k \geq 0$ is an integer and $l$ is a class in $\mathbb{Z} / (q-1)\mathbb{Z}$, if
\begin{enumerate}
\item for $\gamma = 
\left(\begin{smallmatrix} a&b\\ c&d \end{smallmatrix}\right)
\in \Gamma$, $f(\gamma z)=(\det \gamma)^{-l}(cz+d)^kf(z)$;
\item $f$ is rigid analytic on $\Omega$;
\item \label{cusps} $f$ is analytic at the cusps of $\Gamma$.
\end{enumerate}
\end{definition}

Condition (\ref{cusps}) means that for some -- and therefore any -- analytic parameter at the cusp in question, $f$ has a power series expansion in this parameter with positive radius of convergence. In this paper we shall need to discuss the coefficients of such an expansion, which are not independent of the choice of parameter. For the cusp $\infty$, we fix once and for all the function $u$ above as this parameter.

To describe the expansion at the cusp $0$ for a Drinfeld modular form $f$ for $\Gamma_0(\mathfrak{p})$, we first need to define a slash operator: For any $x \in K_{\infty}^{\times}$, $x$ can be written uniquely as 
\begin{equation}\label{leadingcoeff}
x=\zeta_x \left(\frac{1}{T}\right)^{v_{\infty}(x)}u_x
\end{equation}
where $\zeta_x \in \mathbb{F}_q^{\times}$, and $u_x$ is such that $v_{\infty}(u_x-1)>0$, or in other words $u_x$ is a $1$-unit at $\infty$. We call $\zeta_x$ the \emph{leading coefficient} of $x$.

For $\gamma \in \Gl_2(K)$ we have that $\det \gamma \in K^{\times}$. By (\ref{leadingcoeff}), we can write
\begin{equation*}
\det \gamma= \zeta_{\det \gamma} \left(\frac{1}{T}\right)^{v_{\infty}(\det \gamma)}u_{\det \gamma}.
\end{equation*}
For simplicity we write 
\begin{equation*}
\zeta_{\det \gamma}=\zeta_{\gamma}.
\end{equation*}

We define a \emph{slash operator} for $\gamma = 
\left(\begin{smallmatrix} a&b\\ c&d \end{smallmatrix}\right)
\in \Gl_2(K)$ on a modular form of weight $k$ and type $l$ by
\begin{equation}\label{slash}
f|_{k,l}[\gamma]=\zeta_{\gamma}^l \left(\frac{\det \gamma}{\zeta_{\gamma}}\right)^{k/2}(cz+d)^{-k}f(\gamma z).
\end{equation}
Note that for $\gamma \in \Gl_2(A)$ we have that $\det \gamma =\zeta_{\gamma}$; thus if $f$ is modular of weight $k$ and type $l$ for $\Gamma$ and $\gamma \in \Gamma$, then $f|_{k,l}[\gamma]= f$. 

The matrix
\begin{equation*}
W_{\mathfrak{p}}=
\begin{pmatrix}
0 & -1 \\
\pi & 0
\end{pmatrix}
\end{equation*}
sends the cusp $\infty$ to the cusp $0$ on $X_0(\mathfrak{p})$. We define the $u$-series expansion of a Drinfeld modular form $f$ of weight $k$ and type $l$ for $\Gamma_0(\mathfrak{p})$ at $0$ to be that of the form
\begin{equation*}
f|_{k,l} [W_{\mathfrak{p}}] = \pi^{k/2}(\pi z)^{-k}f\left(\frac{-1}{\pi z}\right)
\end{equation*}
at $\infty$. We take this opportunity to note that the operator $|_{k,l}[W_{\mathfrak{p}}]$ is an involution on the Drinfeld modular forms for $\Gamma_0(\mathfrak{p})$ called the \emph{Fricke involution}.

Finally, when we simply speak of the $u$-series expansion of a form, without specifying a cusp, we will always mean the $u$-series expansion at the cusp $\infty$.

\subsection{Modular forms for $\Gl_2(A)$}

We will need a few facts about the algebra of Drinfeld modular forms for the full modular group $\Gl_2(A)$, which we collect here. 

For $k$ a positive integer and $z \in \Omega$, Goss defines in~\cite{gosseisenstein} an Eisenstein series of weight $q^k-1$ by:
\begin{equation}\label{EF}
g_k \defi (-1)^{k+1}\tilde{\pi}^{1-q^k}L_k\sum_{\substack{a,b\in A\\(a,b)\neq(0,0)}}\frac{1}{(az+b)^{q^k-1}},
\end{equation}
where $L_k$ is the least common multiple of all monics of degree $k$, so that 
\begin{equation*}
L_k=(T^q-T)\ldots(T^{q^k}-T),
\end{equation*}
and $\tilde{\pi}$ is the Carlitz period fixed above. These series converge and thus define rigid analytic functions on $\Omega$. They should be considered the analogues of the classical Eisenstein series, and they can be shown to be modular of weight $q^k-1$ and type $0$ for $\Gl_2(A)$. Finally, it is shown in \cite{gekeler} that with this normalization each $g_k$ has integral $u$-series coefficients.

Another modular form for $\Gl_2(A)$ which will be important in this paper is the Poincar\'{e} series of weight $q+1$ and type 1, first defined by Gerritzen and van der Put in~\cite[page 304]{gerritzen}. Let $H$ be the subgroup 
\begin{equation*}
\left\{ 
\begin{pmatrix}
* & *\\
0 & 1
\end{pmatrix}
\right\} \subset \operatorname{GL}_2(A)
\end{equation*}
and as usual
\begin{equation*}
\gamma=
\begin{pmatrix}
a & b\\
c & d
\end{pmatrix} 
\in \operatorname{GL}_2(A).
\end{equation*}
Then we may define a series
\begin{equation}\label{PC}
h \defi \sum_{\gamma \in H\backslash \operatorname{GL}_2(A)}
\frac{\operatorname{det} \gamma \cdot u(\gamma z)}{(cz+d)^{q+1}}.
\end{equation}
Using the properties of the function $u(z)$, this series can be shown to in fact define a Drinfeld modular form of weight $q+1$ and type $1$. It is shown in \cite{gekeler} that $h$ also has integral $u$-series coefficients.

It is a well-known fact (see for example~\cite{gekeler}) that the graded $C$-algebra of Drinfeld modular forms of all weights and all types for $\Gl_2(A)$ is the polynomial ring $C[g_1,h]$ (where each Drinfeld modular form corresponds to a unique isobaric polynomial). Because of the prominent role of the Drinfeld Eisenstein series of weight $q-1$ in the theory, as is customary we will simply write $g$ instead of $g_1$ from now on.

\subsection{$\mathfrak{p}$-adic Drinfeld modular forms for $\Gl_2(A)$}

As before, $\pi(T) \in A$ is a monic prime polynomial of degree $d$ and we denote by $\mathfrak{p}$ the principal ideal that it generates. For $x \in K$, we write $v_{\mathfrak{p}}(x)$ for the valuation of $x$ at $\mathfrak{p}$.

\begin{definition}
Let $f=\sum_{i=0}^{\infty} c_i u^i$ be a formal series with $c_i \in K$. Then we define the \emph{valuation of $f$ at $\mathfrak{p}$} to be
\begin{equation*}
v_{\mathfrak{p}}(f)=\inf_{i} v_{\mathfrak{p}}(c_i).
\end{equation*}
For two formal series $f=\sum a_iu^i$ and $g=\sum b_i u^i$, we write $f\equiv g \pmod{\mathfrak{p}^m}$ if $v_{\mathfrak{p}}(f-g) \geq m$.
\end{definition}

We note that it was shown in \cite{gekeler} that if $\pi(T)$ is of degree $d$, then
\begin{equation*}
g_d \equiv 1 \pmod{\mathfrak{p}}.
\end{equation*}

Following the definition of Serre \cite{serre}, we define

\begin{definition}
A \emph{$\mathfrak{p}$-adic Drinfeld modular form} is a formal $u$-series expansion $\sum a_j u^j$ such that there exists a sequence $\{f_i\}$ of Drinfeld modular forms for $\Gl_2(A)$ such that $v_{\mathfrak{p}}(f_i-f) \to \infty$ as $i \to \infty$.
\end{definition}

We do not know as yet the extent to which the coefficients of these $\mathfrak{p}$-adic Drinfeld modular forms have nice arithmetic properties.

\section{Operators on $\Gamma_0(\mathfrak{p})$}\label{operators}

\subsection{Integrality of $U_{\mathfrak{p}}$ and $V_{\mathfrak{p}}$}

We begin by introducing two operators relevant to the theory of $\mathfrak{p}$-adic Drinfeld modular forms. As before $\mathfrak{p}$ is a prime ideal generated by a monic prime polynomial $\pi(T)$ of $A$ of degree $d$. For any rigid analytic function $f$ on $\Omega$ with expansion $f=\sum_{i=0}^{\infty}c_iu^i$ at $\infty$ we define:
\begin{equation*}
f|U_{\mathfrak{p}}= \frac{1}{\pi} \sum_{\substack{\lambda \in A \\ \degree \lambda < d}} f\left(\frac{z+\lambda}{\pi}\right),
\end{equation*}
and
\begin{equation*}
f|V_{\mathfrak{p}}= f(\pi z).
\end{equation*}

We will show that if the coefficients $c_i$ are integral, then the $u$-series coefficients of $f|U_{\mathfrak{p}}$ and $f|V_{\mathfrak{p}}$ are also integral and moreover that
\begin{equation*}
v_{\mathfrak{p}}(f|U_{\mathfrak{p}})\geq v_{\mathfrak{p}}(f)
\end{equation*}
and
\begin{equation*}
v_{\mathfrak{p}}(f|V_{\mathfrak{p}})\geq v_{\mathfrak{p}}(f).
\end{equation*}

We first consider the operator $U_{\mathfrak{p}}$. This operator was already studied in \cite{bosser}, where the author determined that the $U_{\mathfrak{p}}$ operator acts in the following manner on the coefficients at $\infty$ of analytic functions on $\Omega$ (we note that Bosser's result is more general and applies to meromorphic functions with a pole of order less than $q^d$ at infinity, but we will only need the version stated here):

\begin{proposition}\label{integralityu}
Let $\mathfrak{p}$ be a prime ideal in $A$ generated by a monic prime polynomial $\pi$ of degree $d$, and let $f$ be an analytic function on $\Omega$. Assume that $f$ has a $u$-series expansion of the form
\begin{equation*}
f=\sum_{i=0}^{\infty} c_i u^i, \qquad c_i \in C.
\end{equation*}
As before we write the Carlitz module evaluated at $\pi$ as $\rho_{\pi}(x)=\pi x+\sum_{1\leq i \leq d}\alpha_i x^{q^i}$. Then $f|U_{\mathfrak{p}}$ has a $u$-series expansion
\begin{equation*}
f|U_{\mathfrak{p}}=\sum_{j=1}^{\infty}a_j u^j
\end{equation*}
with
\begin{equation*}
a_j=\sum_{j\leq n \leq 1+(j-1)q^d} \sum_{\substack{i \in \mathbb{N}^{d+1} \\ i_0+i_1+\ldots+i_d=j-1 \\ i_0+i_1q+ \ldots +i_dq^d=n-1}} {{j-1} \choose i} c_n \alpha_1^{i_1}\ldots\alpha_d^{i_d}\pi^{i_0}.
\end{equation*}
\end{proposition}

From this explicit result we deduce that $U_{\mathfrak{p}}$ indeed preserves integrality of the $u$-series coefficients, since each $\alpha_i$ is integral, and furthermore:

\begin{corollary}\label{valuationup}
Suppose that $f$ has $u$-series coefficients in $K$. Then
\begin{equation*}
v_{\mathfrak{p}}(f|U_{\mathfrak{p}})\geq v_{\mathfrak{p}}(f).
\end{equation*}
\end{corollary}
\begin{proof}
This follows from the properties of a non-archimedean valuation, which imply that for each $j$
\begin{equation*}
v_{\mathfrak{p}}(a_j) \geq \min_{j\leq n \leq 1+(j-1)q^d} \{ v_{\mathfrak{p}}(c_n)\}.
\end{equation*}
\end{proof}

We now establish the same properties for the $V_{\mathfrak{p}}$ operator:

\begin{proposition}\label{integralityv}
Let $\mathfrak{p}$ be a prime ideal in $A$ generated by a monic prime polynomial $\pi$ of degree $d$, and let $f$ be an analytic function on $\Omega$ with $u$-series expansion of the form
\begin{equation*}
f=\sum_{i=0}^{\infty} c_i u^i, \qquad c_i \in C.
\end{equation*}
Then if each $c_i$ is integral, then so are the $u$-series coefficients of $f|V_{\mathfrak{p}}$. In addition, if each $c_i \in K$,
\begin{equation*}
v_{\mathfrak{p}}(f|V_{\mathfrak{p}})\geq v_{\mathfrak{p}}(f).
\end{equation*}
\end{proposition}

\begin{proof}
We have:
\begin{equation*}
f|V_{\mathfrak{p}}=\sum_{i=0}^{\infty} c_i u(\pi z)^i,
\end{equation*}
and so we first investigate the $u$-series expansion of $u(\pi z)$.

By definition, if $L=\tilde{\pi}A$ is the lattice associated to the Carlitz module and $e_L(z)$ is the exponential function associated to it, 
\begin{equation*}
e_L(\pi z)=\rho_{\pi}(e_L(z)).
\end{equation*}
We also define the $\pi^{th}$ inverse cyclotomic polynomial
\begin{equation*}
f_{\pi}(X)=X^{q^d}\rho_{\pi}(X^{-1});
\end{equation*}
$f_{\pi}$ is a polynomial with integral coefficients.
Thus we have the straightforward computation:
\begin{align*}
u(\pi z) &= \frac{1}{e_L( \pi ( \tilde{\pi} z ))}\\
	&= \frac{1}{\rho_{\pi}(e_L(\tilde{\pi}z))}\\
	&= \frac{1}{ \rho_{\pi}\left( \frac{1}{ u(z) } \right) }\\
	&= \frac{ u(z)^{q^d} }{f_{\pi}(u(z))}.
\end{align*}
Since $f_{\pi}(0)=1$, the formal expansion in $X$ for 
\begin{equation*}
\frac{ X^{q^d} }{f_{\pi}(X)} = X^{q^d} + \text{ higher order terms.}
\end{equation*}
has integer coefficients, and $u(\pi z)$ has a formal series expansion in $u(z)$ with integral coefficients.

Thus we have
\begin{equation*}
f|V_{\mathfrak{p}}=\sum_{i=0}^{\infty} c_i \left(  \frac{ u(z)^{q^d} }{f_{\pi}(u(z))}\right)^i.
\end{equation*}
We note that for $j$ fixed, only a finite number of terms of the right hand side contribute to the coefficient of $u^j$ on the left hand side, and they are all integral if the $c_i$'s are integral. We conclude that in this case $f|V_{\mathfrak{p}}$ also has integral $u$-series expansion.

Suppose now that the $c_i$'s are merely in $K$, and $v_{\mathfrak{p}}(f)=m$, which implies that $v_{\mathfrak{p}}(c_i)\geq m$ for each $i$. Then each of the summands in the coefficient of $u^j$ for fixed $j$ on the left hand side has valuation $\geq m$. We conclude that the coefficient of $u^j$ also has valuation $\geq m$, which in turns implies that $v_{\mathfrak{p}}(f|V_{\mathfrak{p}})\geq m= v_{\mathfrak{p}}(f)$ and completes the proof.
\end{proof}

We end this section by relating the $V_{\mathfrak{p}}$ operator to the operator $|_{k,l} [W_{\mathfrak{p}}]$ defined by
\begin{equation*}
f|_{k,l} [W_{\mathfrak{p}}] =\pi^{k/2}(\pi z)^{-k}f\left(\frac{-1}{\pi z}\right),
\end{equation*}
as before. We have:

\begin{lemma}\label{integralityw}
Let $f$ be a modular form for $\Gl_2(A)$ of weight $k$ and type $l$. Then
\begin{equation*}
f|_{k,l} [W_{\mathfrak{p}}] = \pi^{k/2} f|V_{\mathfrak{p}}.
\end{equation*}
\end{lemma}

\begin{proof}
We have that 

\begin{equation*}
W_{\mathfrak{p}}=
\begin{pmatrix}
0  & -1 \\
\pi & 0
\end{pmatrix}
=
\begin{pmatrix}
0 & -1 \\
1 & 0
\end{pmatrix}
\begin{pmatrix}
\pi & 0 \\
0 & 1
\end{pmatrix}.
\end{equation*}

So that if we let
\begin{equation*}
S=
\begin{pmatrix}
0 & -1 \\
1 &0
\end{pmatrix}
\in \Gl_2(A), \qquad \text {and} \qquad [\pi]= 
\begin{pmatrix}
\pi & 0 \\
0 & 1
\end{pmatrix},
\end{equation*}
we have
\begin{align*}
f|_{k,l} [W_{\mathfrak{p}}] &= f |_{k,l}[S] |_{k,l}[\pi]\\
	&=f|_{k,l}[\pi] \\
	&= \pi^{k/2}f|V_{\mathfrak{p}},
\end{align*}
where the invariance of $f$ under the action of $|_{k,l}[S] $ follows from the fact that $f$ is modular for the full $\Gl_2(A)$.
\end{proof}

\begin{remark}
From this fact, it follows that the action of $ |_{k,l}[W_{\mathfrak{p}}]$ preserves integrality of the $u$-series coefficients if $f$ is modular for $\Gl_2(A)$.
\end{remark}

\subsection{Norm and trace}

\begin{definition}
For $f$ a modular form of weight $k$ and type $l$ for $\Gamma$ a congruence subgroup of $\Gl_2(A)$, define its \emph{trace} as
\begin{equation*}
\Tr(f)= \sum_{\gamma \in \Gamma \backslash \Gl_2(A)} f |_{k,l}[\gamma].
\end{equation*}
The form $\Tr(f)$ is independent of the choice of coset representatives for $\Gamma \backslash \Gl_2(A)$, and it is a modular form of weight $k$ and type $l$ for $\Gl_2(A)$.
\end{definition}

\begin{definition}
For $f$ a modular form of weight $k$ and type $l$ for $\Gamma$ a congruence subgroup of $\Gl_2(A)$, define its \emph{norm} as
\begin{equation*}
\N(f)= \prod_{\gamma \in \Gamma \backslash \Gl_2(A)} f |_{k,l}[\gamma].
\end{equation*}
Again, $\N(f)$ is independent of the choice of coset representatives for $\Gamma \backslash \Gl_2(A)$, and it is a modular form of weight $k(q^d+1)$ and type $l$ for $\Gl_2(A)$.
\end{definition}

We restrict our attention to the case $\Gamma=\Gamma_0(\mathfrak{p})$, where we have:

\begin{lemma}\label{cosets}
Let $\mathfrak{p}$ be an ideal generated by $\pi(T)$, a monic prime polynomial. The set
\begin{equation*}
\left\{
\begin{pmatrix}
0 & -1 \\
1 & \lambda
\end{pmatrix}
\, | \, \degree \lambda < \degree \pi
\right\},
\end{equation*}
along with the identity, is a complete set of representatives for $\Gamma_0(\mathfrak{p}) \backslash \Gl_2(A)$.
\end{lemma}

The proof of this fact is elementary, and follows as in the classical case.

This explicit set of coset representatives will allow us to give formulae for $\Tr(f)$ and $\N(f)$ in the cases we are interested in in this paper. Before we do this, we introduce a bit of notation to simplify our work below: For $\lambda \in A$ such that $\degree \lambda < \degree \pi$, we write $\gamma_{\lambda}$ for the matrix
\begin{equation*}
\begin{pmatrix}
1/\pi & \lambda/\pi\\
0 & 1
\end{pmatrix}.
\end{equation*}
With this notation, we have
\begin{equation*}
\begin{pmatrix}
0 & -1 \\
1 & \lambda
\end{pmatrix}
= 
\begin{pmatrix}
0 & -1 \\
\pi & 0
\end{pmatrix} \cdot
\begin{pmatrix}
1/\pi& \lambda/\pi \\
0 & 1
\end{pmatrix}
= W_{\mathfrak{p}} \cdot \gamma_{\lambda}.
\end{equation*}

\begin{proposition}\label{traceformula}
Let $f$ be a modular form of weight $k$ and type $l$ for $\Gamma_0(\mathfrak{p})$. Then 
\begin{equation*}
\Tr(f) = f + \pi^{1-k/2} \left( f|_{k,l}[W_{\mathfrak{p}}] \right) | U_{\mathfrak{p}}.
\end{equation*}
\end{proposition}

\begin{proof}
Let
\begin{equation*}
f_0(z)\defi f|_{k,l} [W_{\mathfrak{p}}] = \pi^{k/2}(\pi z)^{-k}f\left(\frac{-1}{\pi z}\right).
\end{equation*}
Since $\zeta_{1/\pi}=1$, and using the coset representatives from Lemma~\ref{cosets} we have
\begin{eqnarray*}
\Tr(f) &=& f + \sum_{\substack{\lambda \in A \\ \degree \lambda <d}} f|_{k,l} \left[
\begin{pmatrix}
0 & -1\\
1 & \lambda
\end{pmatrix} \right] \\
&=& f+  \sum_{\substack{\lambda \in A \\ \degree \lambda <d}} f_0|_{k,l} \left[ \gamma_{\lambda} \right] \\
&=& f+  \sum_{\substack{\lambda \in A \\ \degree \lambda <d}} \left(\frac{1}{\pi}\right)^{k/2}f_0\left(\frac{z+\lambda}{\pi}\right)\\
&=& f + \pi^{1-k/2} f_0 | U_{\mathfrak{p}}.
\end{eqnarray*}
And the result follows from the definition of $f_0$.
\end{proof}

When $f$ is invariant under the Fricke involution, we have

\begin{proposition}\label{normformula}
Let $f$ be a modular form of weight $k$ and type $l$ for $\Gamma_0(\mathfrak{p})$, and suppose furthermore that 
\begin{equation*}
f|_{k,l} [W_{\mathfrak{p}}] = \alpha f,
\end{equation*}
for $\alpha \in \{ \pm 1 \} \subset \mathbb{F}_q^{\times}$. Then
\begin{equation*}
\N(f) = \frac{1}{\pi^{q^dk/2}}\; f \prod_{\substack{\lambda \in A \\ \degree \lambda <d}} f\left( \frac{z+\lambda}{\pi}\right).
\end{equation*}
\end{proposition}

\begin{proof}
This time $f_0=\alpha f$, and so going through the same argument as above we have:
\begin{eqnarray*}
\N(f) &=& f \cdot \prod_{\substack{\lambda \in A \\ \degree \lambda <d}} f|_{k,l} \left[
\begin{pmatrix}
0 & -1\\
1 & \lambda
\end{pmatrix} \right] \\
&=& f \cdot   \prod_{\substack{\lambda \in A \\ \degree \lambda <d}} (\alpha f)|_{k,l} \left[ \gamma_{\lambda} \right] \\
&=& \alpha^{q^d} f \cdot \prod_{\substack{\lambda \in A \\ \degree \lambda <d}} \left(\frac{1}{\pi}\right)^{k/2}f\left(\frac{z+\lambda}{\pi}\right)\\
&=& \frac{1}{ \pi^{q^dk/2}} \; f  \prod_{\substack{\lambda \in A \\ \degree \lambda <d}}  f\left(\frac{z+\lambda}{\pi}\right).
\end{eqnarray*}

\end{proof}

\section{Correspondence and trace}

The ground laid in Section \ref{operators} allows us to show

\begin{theorem}\label{tracemap}
Let $f$ be a modular form of weight $k$ and type $l$ for $\Gamma_0(\mathfrak{p})$, with rational $u$-series coefficients. Then $f$ is a $\mathfrak{p}$-adic Drinfeld modular form for $\Gl_2(A)$.
\end{theorem}

\begin{proof}
For any positive integer $n$ and $g_d$ the Eisenstein series of weight $q^d-1$ and type $0$ for $\Gl_2(A)$, define
\begin{align*}
g_{(0)} & \defi \left(g_d\right)^n- \pi^{n(q^d-1)/2}\left(g_d\right)^n |_{n(q^d-1),0} [W_{\mathfrak{p}}] \\
	& = \left(g_d\right)^n- \pi^{n(q^d-1)}\left(g_d\right)^n | V_{\mathfrak{p}}.
\end{align*}
It is a modular form of weight $n(q^d-1)$ and type 0 for $\Gamma_0(\mathfrak{p})$. Since $\left(g_d\right)^n | V_{\mathfrak{p}}$ has integral coefficients by Proposition~\ref{integralityv} and $g_d \equiv 1 \pmod{\mathfrak{p}}$, we see that $g_{(0)}$ is congruent to 1 modulo $\mathfrak{p}$. Furthermore,
\begin{align*}
g_{(0)}|_{n(q^d-1),0} [W_{\mathfrak{p}}] & = \left(g_d\right)^n |_{n(q^d-1),0} [W_{\mathfrak{p}}] - \pi^{n(q^d-1)/2} \left(g_d\right)^n \\
  &= \pi^{n(q^d-1)/2} \left(g_d\right)^n |V_{\mathfrak{p}} -  \pi^{(q^d-1)/2}(g_d)^n\\
  &=  \pi^{n(q^d-1)/2}(\left(g_d\right)^n  |V_{\mathfrak{p}}  - \left(g_d\right)^n)\\
  & \equiv  0 \pmod{\mathfrak{p}^{n(q^d-1)/2 +1}}.
\end{align*}

The last congruence follows from noticing that
\begin{equation*}
\left(g_d\right)^n  |V_{\mathfrak{p}} =\left( \left(g_d\right)^n  -1\right)|V_{\mathfrak{p}} +1
\end{equation*}
and applying Proposition~\ref{integralityv} to the $u$-series $\left( g_d \right)^n-1$, which has valuation at least $1$, so that 
\begin{equation*}
\left(g_d\right)^n  |V_{\mathfrak{p}}  \equiv \left(g_d\right)^n \pmod{\mathfrak{p}}.
\end{equation*}

With $n$ fixed as before, define $g_{(r)}=(g_{(0)})^{p^r}$. Since $g_{(0)} \equiv 1 \pmod{\mathfrak{p}}$, we have that 
\begin{equation*}
g_{(r)}= (g_{(0)})^{p^r} \equiv 1 \pmod{\mathfrak{p}^{p^r}}.
\end{equation*}
Similarly, because 
\begin{equation*}
g_{(0)} |_{n(q^d-1),0} [W_{\mathfrak{p}}]  \equiv 0 \pmod{\mathfrak{p}^{n(q^d-1)/2 +1}}
\end{equation*}
it follows that 
\begin{equation*}
g_{(r)} |_{p^rn(q^d-1),0} [W_{\mathfrak{p}}]  = (g_{(0)} |_{n(q^d-1),0} [W_{\mathfrak{p}}] )^{p^r} \equiv 0 \pmod{\mathfrak{p}^{np^r(q^d-1)/2 +p^r}}.
\end{equation*}

The function $fg_{(r)}$ is a Drinfeld modular form of weight $k+np^r(q^d-1)$ and type $l$ for $\Gamma_0(\mathfrak{p})$ with rational coefficients. Thus by Proposition~\ref{traceformula}, $\Tr(fg_{(r)})$ is of weight $k+np^r(q^d-1)$ and type $l$ for $\Gl_2(A)$ and we have
\begin{equation*}
\Tr(fg_{(r)}) - f  = (\Tr(fg_{(r)}) -fg_{(r)})+f(g_{(r)}-1).
\end{equation*}

We first bound the valuation at $\mathfrak{p}$ of the term $f(g_{(r)}-1)$ from below, using the fact that  $g_{(r)} \equiv 1 \pmod{\mathfrak{p}^{p^r}}$:
\begin{equation*}
v_{\mathfrak{p}}(f(g_{(r)}-1)) \geq p^r + v_{\mathfrak{p}}(f).
\end{equation*}

We consider now
\begin{equation*}
\Tr(fg_{(r)}) -fg_{(r)}= \pi^{1-(k+np^r(q^d-1))/2} \left( (fg_{(r)})|_{k+np^r(q^d-1),l}[W_{\mathfrak{p}}] \right) | U_{\mathfrak{p}}.
\end{equation*}
Since we have $v_{\mathfrak{p}}(f|U_{\mathfrak{p}})\geq v_{\mathfrak{p}}(f)$, it follows that:
\begin{align*}\label{traceinequality}
v_{\mathfrak{p}}(\Tr(fg_{(r)}) -fg_{(r)})) & \geq  1-(k+np^r(q^d-1))/2 + v_{\mathfrak{p}}\left((fg_{(r)})|_{k+np^r(q^d-1),l}[W_{\mathfrak{p}}] \right) \\
	& = 1-(k+np^r(q^d-1))/2 + v_{\mathfrak{p}}\left(f|_{k,l} [W_{\mathfrak{p}}] \right) + v_{\mathfrak{p}}\left(g_{(r)}|_{np^r(q^d-1),0}[W_{\mathfrak{p}}] \right) \\
	&  = 1-(k+np^r(q^d-1))/2 + v_{\mathfrak{p}}\left(f|_{k,l} [W_{\mathfrak{p}}] \right) + np^r(q^d-1)/2 +p^r\\
	&  = 1-k/2  + v_{\mathfrak{p}}\left(f|_{k,l} [W_{\mathfrak{p}}] \right) +p^r.
\end{align*}

We conclude that
\begin{equation*}
v_{\mathfrak{p}}(\Tr(fg_{(r)}) - f ) \geq \min \{ p^r + v_{\mathfrak{p}}(f), p^r+1-k/2  + v_{\mathfrak{p}}\left(f|_{k,l} [W_{\mathfrak{p}}] \right) \}.
\end{equation*}

Now since $f$ has rational $u$-series coefficients, then so does $f|_{k,l} [W_{\mathfrak{p}}]$, since the Fricke involution of $X_0(\mathfrak{p})$ is defined over the rationals. Thus both $v_{\mathfrak{p}}(f)$ and $v_{\mathfrak{p}}\left(f|_{k,l} [W_{\mathfrak{p}}] \right)$ are finite.

 $\{\Tr(fg_{(r)})\}$ is the sequence of Drinfeld modular forms satisfying the requirements of the definition of a $\mathfrak{p}$-adic Drinfeld modular form.

\end{proof}

As a corollary we can prove Theorem \ref{tracetheorem}:

\begin{proof}
We begin by noting that for $f$ modular of weight $2$ and type $1$ on $\Gamma_0(\mathfrak{p})$, we have that
\begin{equation*}
\Tr(f|_{2,1}[W_{\mathfrak{p}}])=f|_{2,1}[W_{\mathfrak{p}}] + f|U_\mathfrak{p}
\end{equation*}
is a modular form of weight $2$ and type $1$ on $\Gl_2(A)$. However, for all $q\neq2$ this space contains no non-zero modular forms: In any case the full algebra of modular forms for $\Gl_2(A)$ is generated by $g$ and $h$. If $q\geq 4$, then $g$ is of weight $q-1\geq 3$, and so there are no forms of weight $2$. When $q=3$, the forms of weight $2$ are multiples of $g$, which have type $0$, and so there are no non-zero forms of weight $2$ and type $1$.  We conclude that $f|_{2,1}[W_{\mathfrak{p}}] = - f|U_\mathfrak{p}$. Therefore using Corollary \ref{valuationup} and the fact that the Fricke involution is rational, we have that $f|_{2,1}[W_{\mathfrak{p}}]$ has rational $\pi$-integral $u$-series coefficients.

Writing 
\begin{equation*}
g_{(0)}=g_d- \pi^{(q^d-1)/2}g_d |_{q^d-1,0} [W_{\mathfrak{p}}]
\end{equation*}
as before, we consider the map 
\begin{equation*}
f \to \Tr(fg_{(0)}).
\end{equation*}
This map takes an element of the space of forms of weight $2$ and type $1$ for $\Gamma_0(\mathfrak{p})$ to a form of weight $q^d+1$ and type $1$ for $\Gl_2(A)$. We have that $g_d$ has integral coefficients, and since it is a form for $\Gl_2(A)$ so does $g_d|_{q^d-1,0} [W_{\mathfrak{p}}]$ by Lemma \ref{integralityw}. Thus $g_{(0)}$ and $g_{(0)}|_{q^d-1,0}[W_{\mathfrak{p}}]$ have integral $u$-series coefficients. Now from the formula  
\begin{equation*}
\Tr(fg_{(0)}) = fg_{(0)} + \left( f |_{2,1}[W_{\mathfrak{p}}] g_{(0)}|_{q^d-1,0}[W_{\mathfrak{p}}] \right) | U_{\mathfrak{p}},
\end{equation*}
we conclude that $\Tr(fg_{(0)})$ also has rational, $\pi$-integral $u$-series coefficients. Thus the map $f \to \Tr(fg_{(0)})$ preserves rationality and $\pi$-integrality of the $u$-series expansion coefficients.

From the computations in the proof of Theorem \ref{tracemap}, we have
\begin{equation*}
v_{\mathfrak{p}}(\Tr(fg_{(0)}) - f ) \geq \inf \{ 1 + v_{\mathfrak{p}}(f), 1  + v_{\mathfrak{p}}\left(f|_{2,1} [W_{\mathfrak{p}}] \right) \} \geq 1,
\end{equation*}
so that $f \equiv \Tr(fg_{(0)}) \pmod{\mathfrak{p}}$. 

Now consider $\widetilde{M}_{\mathfrak{p},2,1}$ the set of $\tilde{f} \in A/{\mathfrak{p}}[\![u]\!]$ such that there is $f$ of weight $2$ and type $1$ for $\Gamma_0(\mathfrak{p})$ with rational, $\pi$-integral coefficients and $f \equiv \tilde{f} \pmod{\mathfrak{p}}$. The space of forms of weight $2$ and type $1$ for $\Gamma_0(\mathfrak{p})$ is of dimension $g_{\mathfrak{p}}+1,$ where
\begin{equation*}
g_{\mathfrak{p}} \defi
\begin{cases}
\frac{q(q^{d-1}-1)}{q^2-1} & \qquad \mbox{if $d$ is odd,}\\
\frac{q^2(q^{d-2}-1)}{q^2-1} & \qquad \mbox{if $d$ is even,}
\end{cases}
\end{equation*}
and has a basis of forms with integral $u$-series coefficients, and from this it follows that $\widetilde{M}_{\mathfrak{p},2,1}$ has dimension $g_{\mathfrak{p}}+1$ as an $A/{\mathfrak{p}}$-vector space.

Since $f \equiv \Tr(fg_{(0)}) \pmod{\mathfrak{p}}$, $\widetilde{M}_{\mathfrak{p},2,1}$ is a subspace of the $A/{\mathfrak{p}}$-vector space $\widetilde{M}_{q^d+1,1}$, the space that contains the reductions modulo $\mathfrak{p}$ of all of the forms of weight $q^d+1$ and type $1$ for $\Gl_2(A)$ with rational, $\pi$-integral $u$-series coefficients. However, the space $\tilde{M}_{q^d+1,1}$ also has dimension $g_{\mathfrak{p}}+1$, since $M_{q^d+1,1}(\Gl_2(A))$ has a basis of forms with integral $u$-series coefficients. Thus $\widetilde{M}_{\mathfrak{p},2,1}=\widetilde{M}_{q^d+1,1}$, and the trace map establishes a one-to-one correspondence between the spaces, as claimed.
\end{proof}

\begin{remark}
When $q=2$, there are no non-trivial types, and the space of forms of weight $2$ for $\Gl_2(A)$ is spanned by $g^2$. Thus we cannot guarantee that $\Tr(f|_{2,1}[W_{\mathfrak{p}}])=0$, or even that $f|_{2,1}[W_{\mathfrak{p}}]$ has integral coefficients. In the classical case the forms on $\Gamma_0(\ell)$ such that $\Tr(f)=\Tr(f|[W_{\ell}])=0$ are exactly linear combinations of newforms. It is reasonable to conjecture that a similar result holds here and that the existence of oldforms of weight $2$ for $\Gamma_0(\mathfrak{p})$ is exactly the obstruction to the result we seek.
\end{remark}

\section{Norm}

In light of Proposition \ref{normformula}, we begin by studying the product
\begin{equation*}
 \prod_{\substack{\lambda \in A \\ \degree \lambda <d}} f\left( \frac{z+\lambda}{\pi}\right)
\end{equation*}
for $f$ a function on $\Omega$ with $u$-series expansion $\sum_{n=0}^{\infty} a_n u^n$.

\subsection{A combinatorial result}

For this section, let $u$ be an indeterminate. Consider the polynomial
\begin{equation*}
\rho_{\pi}(x) - \frac{1}{u} \in A(\!(u)\!)[x].
\end{equation*}
By equation (\ref{rho}), this is a polynomial of degree $q^d$, and it is relatively prime to its formal derivative, and therefore separable. Denote by 
\begin{equation*}
\{ \gamma_j : 1 \leq j \leq q^d \}
\end{equation*}
the set of roots of this polynomial. We have:

\begin{proposition}\label{normprop}
Let $f(u) \in A[\![ u ]\!]$. Then
\begin{equation*}
\prod_{j=1}^{q^d} f\left(\frac{1}{\gamma_j}\right) \equiv f(u) \pmod{\mathfrak{p}}.
\end{equation*}
Furthermore, if $f(u) = \sum_{n=n_0}^{\infty} a_n u^n$, with $a_n \neq 0$, the leading coefficient of
\begin{equation*}
\prod_{j=1}^{q^d} f\left(\frac{1}{\gamma_j}\right) 
\end{equation*}
is $a_{n_0}^{q^d}$.
\end{proposition}

To obtain this result requires studying symmetric polynomials evaluated at the roots $\gamma_j$. We quickly introduce the notation we will need. We refer the reader to \cite{macdonald} and \cite[Chapter 7]{stanley} for all proofs and further discussion.

Let $x=(x_1, x_2, \ldots)$ be a set of indeterminates, and $\mu=(\mu_1, \mu_2, \ldots)$ be a partition of a positive integer $n$. Then $\mu$ determines a monomial 
\begin{equation*}
x^{\mu} = x_1^{\mu_1} x_2^{\mu_2} \ldots
\end{equation*}
We define the monomial symmetric function
\begin{equation*}
m_{\mu} = \sum_{\alpha} x^{\alpha},
\end{equation*}
where the sum ranges over all distinct permutations $\alpha=(\alpha_1, \alpha_2, \ldots)$ of the entries of the vector $\mu=(\mu_1, \mu_2, \ldots)$. We also set
\begin{equation*}
m_{\emptyset} = 1.
\end{equation*}
When $\mu$ is allowed to range over all partitions of all positive integers, including the empty partition, the set $\{ m_{\mu} \}$ forms a $\mathbb{Z}$-basis of the algebra of symmetric functions.

We also define the elementary symmetric functions
\begin{equation*}
e_{k} = m_{1^k} = \sum_{i_1 < i_2 < \ldots < i_k} x_{i_1} \ldots x_{i_k}
\end{equation*}
for each $k\geq 1$, where $1^k$ is the partition of $k$ whose parts are all equal to $1$, and
\begin{equation*}
e_0= m_{\emptyset} = 1.
\end{equation*}
We further write, again for $\mu=(\mu_1, \mu_2, \ldots)$ a partition,
\begin{equation*}
e_{\mu} = e_{\mu_1} e_{\mu_2} \ldots.
\end{equation*}
Again, the set $\{ e_{\mu} \}$, when $\mu$ ranges over all partitions including the partition of $0$, is a $\mathbb{Z}$-basis of the algebra of symmetric functions.

\begin{lemma}\label{conglemma}
Let $p \in \mathbb{Z}$ be a prime. Let $\mu$ be a partition of the integer $n = m \cdot p^r$, for some positive integers $m$ and $r$, that is not the partition $(m, m, \ldots, m)$ (where $m$ appears $p^r$ times). Let $\xi$ be the partition of $n$ whose parts are all equal to $p^r$. Then when we write
\begin{equation*}
m_{\mu} = \sum_{\nu \,  \vdash \,  mp^r} a_{\mu \nu} e_{\nu},
\end{equation*}
the coefficient $a_{\mu \xi}$ appearing before the elementary symmetric function $e_{\xi}$ is divisible by $p$.
\end{lemma}

\begin{proof}
Throughout we fix a partition $\mu$ satisfying the hypotheses of the lemma. If $\mu'$ is the conjugate partition of $\mu$, we can write
\begin{equation*}
e_{\mu'} = m_{\mu} + \sum_{\nu < \mu} b_{\mu \nu} m_{\nu},
\end{equation*}
where $<$ denotes the natural order on the set of partitions of $n$, and each $b_{\mu \nu}$ is a nonnegative integer.

Solving for $m_{\mu}$ we obtain
\begin{equation*}
m_{\mu} = e_{\mu'} - \sum_{\nu < \mu} b_{\mu \nu} m_{\nu}.
\end{equation*}
To write the elements of the monomial basis $m$ in terms of the elements of the elementary basis $e$, it suffices to ``back-solve" repeatedly:
\begin{equation*}
m_{\mu} = e_{\mu'} - \sum_{\nu < \mu} b_{\mu \nu} \left(e_{\nu'} - \sum_{\kappa < \nu} b_{\nu \kappa} m_{\kappa}\right),
\end{equation*}
and so on. Proceeding in this manner, it suffices to show that $b_{\mu \xi'}$ is divisible by $p$ for any partition $\mu$ that is not $\xi'$ to ensure that $a_{\mu \xi}$ is divisible by $p$. (Note that we have excluded the case where $\mu = \xi'$ in the hypotheses of the lemma: $\xi'=(m, \ldots, m)$, where $m$ appears $p^r$ times.)

Since for any two partitions $\nu$ and $\kappa$, we have 
\begin{equation*}
b_{\nu \kappa} = b_{\kappa' \nu'},
\end{equation*}
we can instead show that $b_{\xi \nu}$ is divisible by $p$ for any partition $\nu$ of $n$ which is not the partition $\xi$.

The coefficients $b_{\xi \nu}$ appear in the expansion of $e_{\xi'}$ in terms of the monomial basis $m$. But
\begin{equation*}
e_{\xi'} = (e_m )^{p^r},
\end{equation*}
by definition. Furthermore, $e_m= m_{1^m}$, again by definition. Because we are raising to a $p^{th}$ power, all cross-terms disappear modulo $p$, and we are left with the congruence
\begin{equation*}
e_{\xi'} \equiv m_{\xi} \pmod{p}.
\end{equation*}

This shows that $b_{\xi \nu}$ is divisible by $p$ for any partition $\nu$ of $n$ which is not the partition $\xi$, and completes the proof.
\end{proof}

With this lemma, we can now prove Proposition \ref{normprop}:

\begin{proof}[Proof of Proposition \ref{normprop}]
Write
\begin{equation*}
f(u) = \sum_{n=0}^{\infty}a_nu^n \in A[\![ u ]\!].
\end{equation*}

Then the product 
\begin{equation*}
\prod_{j=1}^{q^d} f\left(\frac{1}{\gamma_j}\right) 
\end{equation*}
is a sum of terms of the form
\begin{equation*}
a_{n_1}^{i_1}a_{n_2}^{i_2} \ldots a_{n_k}^{i_k} \left( \sum_{\alpha} \frac{1}{\gamma^{\alpha} }\right),
\end{equation*}
where each $i_j$ is a positive integer and the $n_j$'s are distinct, $\sum_{j=1}^{k} i_j = q^d$,  $\alpha=(\alpha_1, \alpha_2, \ldots, \alpha_{q^d})$ ranges over all distinct permutations of the entries of the vector  $n=(n_1^{i_1}, n_2^{i_2}, \ldots n_k^{i_k})$ (the exponents here denote an entry that is repeated $i_j$ times), and as before $\gamma^{\alpha}=\gamma_1^{\alpha_1}\gamma_2^{\alpha_2}\ldots$

If $k=1$, i.e. $n=(m^{q^d})$ for some $m$, then 
\begin{equation*}
a_{m}^{q^d} \sum_{\alpha} \frac{1}{\gamma^{\alpha} } = a_{m}^{q^d} \frac{1}{\prod_{j=1}^{q^d} \gamma_j^m}.
\end{equation*}
Since $a_m \in A$, we have that $a_m^{q^d} \equiv a_m \pmod{\mathfrak{p}}$. 
Because the $\gamma_j$'s are all of the roots of the polynomial
\begin{equation*}
\rho_{\pi}(x) - \frac{1}{u},
\end{equation*}
whose constant term is $\frac{1}{u}$, it follows that
\begin{equation*}
\prod_{j=1}^{q^d} \gamma_j = \frac{1}{u}.
\end{equation*}
(This is true even when $q$ is a power of $2$, since in that case $(-1)^{q^d}=1=-1$ in $\mathbb{F}_q$.)
As a consequence we have
\begin{equation*}
a_{m}^{q^d} \sum_{\alpha} \frac{1}{\gamma^{\alpha} }= a_{m}^{q^d} \frac{1}{\prod_{j=1}^{q^d} \gamma_j^m}\equiv a_m u^m \pmod{\mathfrak{p}}.
\end{equation*}

We now study the sum
\begin{equation*}
 \sum_{\alpha} \frac{1}{\gamma^{\alpha} }
\end{equation*}
for a fixed set of indices $n=(n_1^{i_1}, n_2^{i_2}, \ldots n_k^{i_k})$, where $k >1$. 

Let $N = \max_{j} n_j$, then 
\begin{equation*}
 \sum_{\alpha} \frac{1}{\gamma^{\alpha} } = \sum_{\beta} \frac{\gamma^{\beta}}{\prod_{j=1}^{q^d} \gamma_j^N},
\end{equation*}
where $\beta$ ranges over all distinct permutations of the vector 
\begin{equation*}
((N-n_1)^{i_1}, (N-n_2)^{i_2}, \ldots, (N-n_k)^{i_k}).
\end{equation*}
Arranging the entries of this vector in decreasing order and discarding all of the entries equal to $0$, we obtain a partition $\mu$ of the integer $M=Nq^d- \sum_{j}n_j$, and this partition has strictly fewer than $q^d$ parts, but it is not the empty partition. Then we have
\begin{equation*}
 \sum_{\alpha} \frac{1}{\gamma^{\alpha} } =  \frac{m_{\mu}(\gamma)}{\prod_{j=1}^{q^d} \gamma_j^N},
\end{equation*}
where $m_{\mu}(\gamma)$ denotes the monomial symmetric function introduced above evaluated at $x_j=\gamma_j$, $1\leq j \leq q^d$. (As is customary in the study of symmetric functions, we let $x_j=0$ when $j> q^d$.)


We will show that 
\begin{equation*}
m_{\mu}(\gamma) \equiv 0 \pmod{\mathfrak{p}}, 
\end{equation*}
when $\mu$ is a partition of any integer that has strictly fewer than $q^d$ parts. This will complete the proof of the proposition, since then the only terms in the product
\begin{equation*}
\prod_{j=1}^{q^d} f\left(\frac{1}{\gamma_j}\right)
\end{equation*}
that do not vanish modulo $\mathfrak{p}$ will be the ones of the form  
\begin{equation*}
a_{m}^{q^d} \sum_{\alpha} \frac{1}{\gamma^{\alpha} },
\end{equation*}
which were discussed above.

By Corollary \ref{congprop}, we have
\begin{equation*}
\rho_{\pi}(x) - \frac{1}{u} \equiv x^{q^d} - \frac{1}{u} \pmod{\mathfrak{p}}.
\end{equation*}
As a consequence, for $1 \leq r \leq q^d-1$, we have
\begin{equation*}
e_r(\gamma) \equiv 0 \pmod{\mathfrak{p}},
\end{equation*}
where again we use this notation to denote the elementary symmetric function introduced above evaluated at $x_j=\gamma_j$, $1\leq j \leq q^d$. Again, as is customary in the study of symmetric functions, we let $x_j=0$ when $j>q^d$, so that 
\begin{equation*}
e_r(\gamma) = 0
\end{equation*}
when $r > q^d$.

Because of this, for any partition $\mu$ that contains at least one part unequal to $q^d$, 
\begin{equation*}
e_{\mu}(\gamma) \equiv 0 \pmod{\mathfrak{p}}.
\end{equation*}

Let now $\mu$ be a partition of $n$. If $n \not \equiv 0 \pmod{q^d}$, then writing 
\begin{equation*}
m_{\mu} = \sum_{\nu \,  \vdash \,  n} a_{\mu \nu} e_{\nu},
\end{equation*}
we have that each $\nu$ has at least one part that is not equal to $q^d$, so
\begin{equation*}
m_{\mu}(\gamma) \equiv 0 \pmod{\mathfrak{p}}.
\end{equation*}

Let us consider now the case where $n \equiv 0 \pmod{q^d}$, say $n = m \cdot q^d$. Then $n$ has exactly one partition $\xi$ whose parts are all equal to $q^d$, so
\begin{equation*}
m_{\mu}(\gamma) = \sum_{\nu \,  \vdash \,  m} a_{\mu \nu} e_{\nu}(\gamma) \equiv a_{\mu \xi} e_{\xi}(\gamma) \pmod{\mathfrak{p}}.
\end{equation*}
Since $\mu$ has strictly fewer than $q^d$ parts, $\mu$ is not the partition $(m, \ldots, m)$, with $m$ appearing $q^d$ times, so we may apply Lemma \ref{conglemma} to conclude that $a_{\mu \xi}=0$ in $\mathbb{F}_q$. Therefore
\begin{equation*}
m_{\mu}(\gamma) \equiv  0 \pmod{\mathfrak{p}},
\end{equation*}
and this completes the proof of the theorem.
\end{proof}

\subsection{Specializing at $z \in \Omega$}

Let now $z \in \Omega$ be fixed. The polynomial
\begin{equation*}
\rho_{\pi}(x) - \frac{1}{u(z)}, 
\end{equation*}
which is the specialization at $z \in \Omega$ of the polynomial studied above, is still of degree $q^d$, this time in $A(\!(u(z))\!)[x]$.

\begin{lemma}
Fix $z \in \Omega$. The set
\begin{equation*}
\left \{ u\left(\frac{z+\lambda}{\pi}\right) :  \lambda \in A, \degree \lambda < d \right \}
\end{equation*}
is exactly the set of the reciprocal of the roots of the polynomial
\begin{equation*}
\rho_{\pi}(x) - \frac{1}{u(z)}.
\end{equation*}
\end{lemma}

\begin{proof}
The set 
\begin{equation*}
\left \{ u\left(\frac{z+\lambda}{\pi}\right) :  \lambda \in A, \degree \lambda < d \right \}
\end{equation*}
has cardinality $q^d$: Indeed, recall that $u(z) = \frac{1}{e_L(\tilde{\pi}z)}$, for $L$ the lattice associated to the Carlitz module. Then if there are $\lambda_1$ and $\lambda_2$, each with degree less than $d$, such that
\begin{equation*}
u\left(\frac{z+\lambda_1}{\pi}\right) = u\left(\frac{z+\lambda_2}{\pi}\right),
\end{equation*}
it follows that
\begin{equation*}
e_L\left(\tilde{\pi}\left(\frac{\lambda_1-\lambda_2}{\pi}\right)\right) = 0.
\end{equation*}
By definition of $e_L$ this forces 
\begin{equation*}
\frac{\lambda_1-\lambda_2}{\pi} \in A,
\end{equation*}
which because of degree considerations can only be the case if $\lambda_1 = \lambda_2$.

We also have that
\begin{align*}
\rho_{\pi} \left( e_L\left(\frac{\tilde{\pi}(z+\lambda)}{\pi} \right) \right)& = e_L(\tilde{\pi}z + \tilde{\pi}\lambda) \\
& = e_L(\tilde{\pi}z) + e_L( \tilde{\pi}\lambda) \\
& = e_L(\tilde{\pi}z) =  \frac{1}{u(z)}.
\end{align*}
Thus we conclude that 
\begin{equation*}
\frac{1}{u\left(\frac{z+\lambda}{\pi}\right)},
\end{equation*}
as  $\lambda$ ranges over all elements of $A$ of degree less than $d$, is a set of $q^d$ distinct roots of the polynomial $\rho_{\pi}(x) - \frac{1}{u(z)}$, a polynomial of degree $q^d$. This concludes the proof.
\end{proof}

To simplify the notation and be consistent with the notation of the previous section, we will index the roots of
\begin{equation*}
\rho_{\pi}(x) - \frac{1}{u(z)} 
\end{equation*}
using natural numbers:
\begin{equation*}
\{ \gamma_j(z) : 1 \leq j \leq q^d \}.
\end{equation*}

Since the $\gamma_j(z)$ satisfy a polynomial with coefficients in $A$, we have that 
\begin{equation*}
\prod_{\substack{\lambda \in A \\ \degree \lambda <d}} f\left( \frac{z+\lambda}{\pi}\right) \in A[\![u]\!].
\end{equation*}
Furthermore, if $f$ has leading coefficient $1$, then the leading coefficient of 
\begin{equation*}
\prod_{\substack{\lambda \in A \\ \degree \lambda <d}} f\left( \frac{z+\lambda}{\pi}\right)
\end{equation*}
is 1.

As a consequence of Proposition \ref{normprop}, specializing at $z \in \Omega$ we have:
\begin{corollary}
Let $f$ be a function on $\Omega$ with $u$-series expansion 
\begin{equation*}
\sum_{n=0}^{\infty} a_n u^n \in A[\![ u ]\!].
\end{equation*}
Then
\begin{equation*}
 f(z) \prod_{\substack{\lambda \in A \\ \degree \lambda <d}} f\left( \frac{z+\lambda}{\pi}\right) \equiv f(z)^2 \pmod{\mathfrak{p}}.
\end{equation*}
\end{corollary}

With this we may now prove Theorem \ref{normtheorem}:

\begin{proof}
Let $f$ be a Drinfeld modular form for $\Gamma_0(\mathfrak{p})$ that is an eigenform of the Fricke involution. By Proposition \ref{normformula}, 
\begin{equation*}
\N(f) = \frac{1}{\pi^{q^dk/2}}\; f \prod_{\substack{\lambda \in A \\ \degree \lambda <d}} f\left( \frac{z+\lambda}{\pi}\right).
\end{equation*}

Since $f$ has leading coefficient $1$, so does
\begin{equation*}
f \prod_{\substack{\lambda \in A \\ \degree \lambda <d}} f\left( \frac{z+\lambda}{\pi}\right).
\end{equation*}
It follows that 
\begin{equation*}
\widetilde{\N(f)} = f \prod_{\substack{\lambda \in A \\ \degree \lambda <d}} f\left( \frac{z+\lambda}{\pi}\right),
\end{equation*}
and the result follows.

\end{proof}

\bibliography{bibliography} {}
\bibliographystyle{amsplain}
\end{document}